\documentclass[11pt,a4paper]{amsart}
\usepackage{mathrsfs}
\usepackage{syntonly}
\usepackage{amsmath}
\usepackage{amsthm}
\usepackage{amsfonts}
\usepackage{amssymb}
\usepackage{latexsym}
\usepackage{amscd,amssymb,amsopn,amsmath,amsthm,graphics,amsfonts,mathrsfs,accents,enumerate,verbatim,calc}
\usepackage[dvips]{graphicx}
\usepackage[colorlinks=true,linkcolor=red,citecolor=blue]{hyperref}
\usepackage[all]{xy}
\usepackage{enumitem}

\date{}
\allowdisplaybreaks[4] \footskip=15pt
\renewcommand{\uppercasenonmath}[1]{}

\numberwithin{equation}{section} \theoremstyle{plain}
\newtheorem{lem}{Lemma}[section]
\newtheorem{cor}[lem]{Corollary}
\newtheorem{prop}[lem]{Proposition}
\newtheorem{thm}[lem]{Theorem}

\newtheorem{definition}[lem]{Definition}
\newtheorem{Ex}[lem]{Example}
\newtheorem{Quest}[lem]{Question}
\newtheorem{Property}[lem]{Property}
\newtheorem{Properties}[lem]{Properties}
\newtheorem{Subprops}{}[lem]
\newtheorem{Para}[lem]{}

\newtheorem{rem}[lem]{Remark}

\newtheorem*{ack*}{ACKNOWLEDGEMENTS}

\newcommand{\pf}{\noindent\begin {proof}}
\newcommand{\epf}{\end{proof}}

\newcommand{\IE}{\mathcal{\mathcal{I}(\mathcal{E})}}

\begin{document}
\begin{center}
{\Large  \bf  Gorenstein homological dimensions for extriangulated categories}

\vspace{0.5cm}  Jiangsheng Hu$^{a}$, Dongdong Zhang$^{b}$\footnote{Corresponding author. Jiangsheng Hu was supported by the NSF of China (Grants No. 11501257, 11671069, 11771212) and Qing Lan Project of Jiangsu Province. Panyue Zhou was supported by the Hunan Provincial Natural Science Foundation of China (Grant
No. 2018JJ3205) and the NSF of China (Grants No. 11671221)} and Panyue Zhou$^{c}$ \\
\medskip

\hspace{-4mm}$^{a}$School of Mathematics and Physics, Jiangsu University of Technology,
 Changzhou 213001, China\\
 $^b$Department of Mathematics, Zhejiang Normal University,
\small Jinhua 321004, China\\
$^c$College of Mathematics, Hunan Institute of Science and Technology, Yueyang, Hunan 414006, China\\
E-mails: jiangshenghu@jsut.edu.cn, zdd@zjnu.cn and panyuezhou@163.com \\
\end{center}

\bigskip
\centerline { \bf  Abstract}
\medskip

\leftskip10truemm \rightskip10truemm \noindent
\hspace{1em} Let $(\mathcal{C},\mathbb{E},\mathfrak{s})$ be an extriangulated category with a proper class $\xi$ of
$\mathbb{E}$-triangles. The authors introduced and studied
 $\xi$-$\mathcal{G}$projective and  $\xi$-$\mathcal{G}$injective in \cite{HZZ}.
 In this paper, we  discuss Gorenstein homological
dimensions for extriangulated categories. More precisely, we first give some characterizations of $\xi$-$\mathcal{G}$projective dimension by using derived functors on $\mathcal{C}$. Second, let $\mathcal{P}(\xi)$ (resp. $\mathcal{I}(\xi)$) be a generating (resp. cogenerating) subcategory of $\mathcal{C}$. We show that the following equality holds under some assumptions:
$$\sup\{\xi\textrm{-}\mathcal{G}{\rm pd}M \ | \ \textrm{for} \ \textrm{any} \ M\in{\mathcal{C}}\}=\sup\{\xi\textrm{-}\mathcal{G}{\rm id}M \ | \ \textrm{for} \ \textrm{any} \ M\in{\mathcal{C}}\},$$
where $\xi\textrm{-}\mathcal{G}{\rm pd}M$ (resp. $\xi\textrm{-}\mathcal{G}{\rm id}M$) denotes $\xi$-$\mathcal{G}$projective (resp. $\xi$-$\mathcal{G}$injective) dimension of $M$.
As an application, our main results generalize their work by Bennis-Mahdou and Ren-Liu.
Moreover, our proof is not far from the usual module or triangulated case.\\[2mm]

\leftskip0truemm \rightskip0truemm
\section { \bf Introduction}
The notion of extriangulated categories was introduced by Nakaoka and Palu in \cite{NP} as a simultaneous generalization of
exact categories and triangulated categories. Exact categories and extension closed subcategories of an
extriangulated category are extriangulated categories, while there exist some other examples of extriangulated categories which are neither exact nor triangulated, see \cite{NP,ZZ,HZZ}. Hence many results hold on exact categories
and triangulated categories can be unified in the same framework.

Let $(\mathcal{C},\mathbb{E},\mathfrak{s})$  be an extriangulated category.  The authors \cite{HZZ} studied a relative homological algebra in $\mathcal C$ which parallels the relative homological algebra in a triangulated category. By specifying a class of $\mathbb{E}$-triangles, which is called a proper class $\xi$ of
$\mathbb{E}$-triangles, we introduced $\xi$-$\mathcal{G}$projective dimensions and  $\xi$-$\mathcal{G}$injective dimensions,
and discussed their properties.

Bennis and Mahdou \cite{BM} proved that the global Gorenstein projective dimension of a ring $R$
is equal to  the global Gorenstein injective dimension of $R$.
Ren and Liu \cite{RL1} studied Gorenstein homological dimensions for triangulated categories, let $\mathcal C$ be a triangulated
category with enough $\xi$-projectives and enough $\xi$-injectives for a fixed proper class of
triangles $\xi$. They showed the following. Assume that the full subcategory $\mathcal P(\xi)$ of
$\xi$-projective objects is a generating subcategory of $\mathcal C$ and the full subcategory $\mathcal I(\xi)$ of
$\xi$-injective objects is a congenerating subcategory of $\mathcal C$, there exists an equality
$${\rm sup}\{\xi\emph{-}{\mathcal G}{\rm pd} M~|~\textrm{for any}~ M\in \mathcal C\}={\rm sup}\{\xi\emph{-}{\mathcal G}{\rm id}M~|~\textrm{for any}~ M\in \mathcal C\},$$
where $\xi$-${\mathcal G}{\rm pd} M$ (resp. $\xi$-${\mathcal G}{\rm id}M$) denotes $\xi$-${\mathcal G}$projective (resp. $\xi$-${\mathcal G}$injective) dimension of $M$.
Motivated by this idea,
in this paper, we study Gorenstein homological
dimensions for extriangulated categories. More precisely, we give some characterizations of $\xi$-$\mathcal{G}$projective dimension by using derived functors on $\mathcal{C}$, see Theorem \ref{prop:2.7}.
In addition, let $\mathcal{P}(\xi)$ (resp. $\mathcal{I}(\xi)$) be a generating (resp. cogenerating) subcategory of $\mathcal{C}$. We show that the following statements are equivalent for some non-negative integer $m$ under some assumptions:
\begin{enumerate}
\item $\sup\{\xi\textrm{-}\mathcal{G}{\rm pd}M \ | \ \textrm{for} \ \textrm{any} \ M\in{\mathcal{C}}\}\leq m$;

\item $\sup\{\xi\textrm{-}\mathcal{G}{\rm id}M \ | \ \textrm{for} \ \textrm{any} \ M\in{\mathcal{C}}\}\leq m$;

\item $\xi\textrm{-}{\rm silp}\mathcal{C}=\xi\textrm{-}{\rm spli}\mathcal{C}\leq m$.
\end{enumerate}
where $\xi\textrm{-}\mathcal{G}{\rm pd}M$ (resp. $\xi\textrm{-}\mathcal{G}{\rm id}M$) denotes $\xi$-$\mathcal{G}$projective (resp. $\xi$-$\mathcal{G}$injective) dimension of $M$ and
$\xi\textrm{-}{\rm silp}\mathcal{C}=\sup\{\xi\textrm{-}{\rm id}P~|~P\in{\mathcal{P}(\xi)}\}$,
$\xi\textrm{-}{\rm spli}\mathcal{C}=\sup\{\xi\textrm{-}{\rm pd}I~|~I\in{\mathcal{I}(\xi)}\}$, see Theorem \ref{thm:3.4}. As a consequence, we have the following equality on $\mathcal{C}$:
$$\sup\{\xi\textrm{-}\mathcal{G}{\rm pd}M \ | \ \textrm{for} \ \textrm{any} \ M\in{\mathcal{C}}\}=\sup\{\xi\textrm{-}\mathcal{G}{\rm id}M \ | \ \textrm{for} \ \textrm{any} \ M\in{\mathcal{C}}\}.$$
Note that module categories and triangulated categories can be viewed as extriangulated categories. As an application, our main results generalize their work by Bennis-Mahdou and Ren-Liu.

\section{\bf Preliminaries}
Let us briefly recall some definitions and basic properties of extriangulated categories from \cite{NP}.
We omit some details here, but the reader can find
them in \cite{NP}.

Let $\mathcal{C}$ be an additive category equipped with an additive bifunctor
$$\mathbb{E}: \mathcal{C}^{\rm op}\times \mathcal{C}\rightarrow {\rm Ab},$$
where ${\rm Ab}$ is the category of abelian groups. For any objects $A, C\in\mathcal{C}$, an element $\delta\in \mathbb{E}(C,A)$ is called an $\mathbb{E}$-extension.
Let $\mathfrak{s}$ be a correspondence which associates an equivalence class $$\mathfrak{s}(\delta)=\xymatrix@C=0.8cm{[A\ar[r]^x
 &B\ar[r]^y&C]}$$ to any $\mathbb{E}$-extension $\delta\in\mathbb{E}(C, A)$. This $\mathfrak{s}$ is called a {\it realization} of $\mathbb{E}$, if it makes the diagrams in \cite[Definition 2.9]{NP} commutative.
 A triplet $(\mathcal{C}, \mathbb{E}, \mathfrak{s})$ is called an {\it extriangulated category} if it satisfies the following conditions.
\begin{enumerate}
\item $\mathbb{E}\colon\mathcal{C}^{\rm op}\times \mathcal{C}\rightarrow \rm{Ab}$ is an additive bifunctor.

\item $\mathfrak{s}$ is an additive realization of $\mathbb{E}$.

\item $\mathbb{E}$ and $\mathfrak{s}$  satisfy the compatibility conditions in \cite[Definition 2.12]{NP}.

 \end{enumerate}

\begin{rem}
Note that both exact categories and triangulated categories are extriangulated categories, see \cite[Example 2.13]{NP} and extension closed subcategories of extriangulated categories are
again extriangulated, see \cite[Remark 2.18]{NP}. Moreover, there exist extriangulated categories which
are neither exact categories nor triangulated categories, see \cite[Proposition 3.30]{NP}, \cite[Example 4.14]{ZZ} and \cite[Remark 3.3]{HZZ}.
\end{rem}

A class of $\mathbb{E}$-triangles $\xi$ is {\it closed under base change} if for any $\mathbb{E}$-triangle $$\xymatrix@C=2em{A\ar[r]^x&B\ar[r]^y&C\ar@{-->}[r]^{\delta}&\in\xi}$$ and any morphism $c\colon C' \to C$, then any $\mathbb{E}$-triangle  $\xymatrix@C=2em{A\ar[r]^{x'}&B'\ar[r]^{y'}&C'\ar@{-->}[r]^{c^*\delta}&}$ belongs to $\xi$.

Dually, a class of  $\mathbb{E}$-triangles $\xi$ is {\it closed under cobase change} if for any $\mathbb{E}$-triangle $$\xymatrix@C=2em{A\ar[r]^x&B\ar[r]^y&C\ar@{-->}[r]^{\delta}&\in\xi}$$ and any morphism $a\colon A \to A'$, then any $\mathbb{E}$-triangle  $\xymatrix@C=2em{A'\ar[r]^{x'}&B'\ar[r]^{y'}&C\ar@{-->}[r]^{a_*\delta}&}$ belongs to $\xi$.

\begin{lem}\label{lem1} \emph{(see \cite[Proposition 3.15]{NP})} Assume that $(\mathcal{C}, \mathbb{E},\mathfrak{s})$ is an extriangulated category. Let $C$ be any object, and let $\xymatrix@C=2em{A_1\ar[r]^{x_1}&B_1\ar[r]^{y_1}&C\ar@{-->}[r]^{\delta_1}&}$ and $\xymatrix@C=2em{A_2\ar[r]^{x_2}&B_2\ar[r]^{y_2}&C\ar@{-->}[r]^{\delta_2}&}$ be any pair of $\mathbb{E}$-triangles. Then there is a commutative diagram
in $\mathcal{C}$
$$\xymatrix{
    & A_2\ar[d]_{m_2} \ar@{=}[r] & A_2 \ar[d]^{x_2} \\
  A_1 \ar@{=}[d] \ar[r]^{m_1} & M \ar[d]_{e_2} \ar[r]^{e_1} & B_2\ar[d]^{y_2} \\
  A_1 \ar[r]^{x_1} & B_1\ar[r]^{y_1} & C   }
  $$
  which satisfies $\mathfrak{s}(y^*_2\delta_1)=\xymatrix@C=2em{[A_1\ar[r]^{m_1}&M\ar[r]^{e_1}&B_2]}$ and
  $\mathfrak{s}(y^*_1\delta_2)=\xymatrix@C=2em{[A_2\ar[r]^{m_2}&M\ar[r]^{e_2}&B_1]}$.

\end{lem}

A class of $\mathbb{E}$-triangles $\xi$ is called {\it saturated} if in the situation of Lemma \ref{lem1}, whenever the third vertical and the second horizontal $\mathbb{E}$-triangles belong to $\xi$, then the  $\mathbb{E}$-triangle $$\xymatrix@C=2em{A_1\ar[r]^{x_1}&B_1\ar[r]^{y_1}&C\ar@{-->}[r]^{\delta_1 }&}$$  belongs to $\xi$.

An $\mathbb{E}$-triangle $\xymatrix@C=2em{A\ar[r]^x&B\ar[r]^y&C\ar@{-->}[r]^{\delta}&}$ is called {\it split} if $\delta=0$. It is easy to see that it is split if and only if $x$ is section or $y$ is retraction. The full subcategory  consisting of the split $\mathbb{E}$-triangles will be denoted by $\Delta_0$.

  \begin{definition} \cite[Definition 3.1]{HZZ}\label{def:proper class}{\rm  Let $\xi$ be a class of $\mathbb{E}$-triangles which is closed under isomorphisms. $\xi$ is called a {\it proper class} of $\mathbb{E}$-triangles if the following conditions hold:

  (1) $\xi$ is closed under finite coproducts and $\Delta_0\subseteq \xi$.

  (2) $\xi$ is closed under base change and cobase change.

  (3) $\xi$ is saturated.}
  \end{definition}
 \begin{definition} \cite[Definition 4.1]{HZZ}
 {\rm An object $P\in\mathcal{C}$  is called {\it $\xi$-projective}  if for any $\mathbb{E}$-triangle $$\xymatrix{A\ar[r]^x& B\ar[r]^y& C \ar@{-->}[r]^{\delta}& }$$ in $\xi$, the induced sequence of abelian groups $\xymatrix@C=0.6cm{0\ar[r]& \mathcal{C}(P,A)\ar[r]& \mathcal{C}(P,B)\ar[r]&\mathcal{C}(P,C)\ar[r]& 0}$ is exact. Dually, we have the definition of {\it $\xi$-injective}.}
\end{definition}

We denote $\mathcal{P(\xi)}$ (respectively $\mathcal{I(\xi)}$) the class of $\xi$-projective (respectively $\xi$-injective) objects of $\mathcal{C}$. It follows from the definition that this subcategory $\mathcal{P}(\xi)$ and $\mathcal{I}(\xi)$ are full, additive, closed under isomorphisms and direct summands.

 An extriangulated  category $(\mathcal{C}, \mathbb{E}, \mathfrak{s})$ is said to  have {\it  enough
$\xi$-projectives} \ (respectively {\it  enough $\xi$-injectives}) provided that for each object $A$ there exists an $\mathbb{E}$-triangle $\xymatrix@C=2.1em{K\ar[r]& P\ar[r]&A\ar@{-->}[r]& }$ (respectively $\xymatrix@C=2em{A\ar[r]& I\ar[r]& K\ar@{-->}[r]&}$) in $\xi$ with $P\in\mathcal{P}(\xi)$ (respectively $I\in\mathcal{I}(\xi)$).

The {\it $\xi$-projective dimension} $\xi$-${\rm pd} A$ of $A\in\mathcal{C}$ is defined inductively.
 If $A\in\mathcal{P}(\xi)$, then define $\xi$-${\rm pd} A=0$.
Next if $\xi$-${\rm pd} A>0$, define $\xi$-${\rm pd} A\leqslant n$ if there exists an $\mathbb{E}$-triangle
 $K\to P\to A\dashrightarrow$  in $\xi$ with $P\in \mathcal{P}(\xi)$ and $\xi$-${\rm pd} K\leqslant n-1$.
Finally we define $\xi$-${\rm pd} A=n$ if $\xi$-${\rm pd} A\leqslant n$ and $\xi$-${\rm pd} A\nleq n-1$. Of course we set $\xi$-${\rm pd} A=\infty$, if $\xi$-${\rm pd} A\neq n$ for all $n\geqslant 0$.

Dually we can define the {\it $\xi$-injective dimension}  $\xi$-${\rm id} A$ of an object $A\in\mathcal{C}$.

\begin{definition} \cite[Definition 4.4]{HZZ}
{\rm An {\it $\xi$-exact} complex $\mathbf{X}$ is a diagram $$\xymatrix@C=2em{\cdots\ar[r]&X_1\ar[r]^{d_1}&X_0\ar[r]^{d_0}&X_{-1}\ar[r]&\cdots}$$ in $\mathcal{C}$ such that for each integer $n$, there exists an $\mathbb{E}$-triangle $\xymatrix@C=2em{K_{n+1}\ar[r]^{g_n}&X_n\ar[r]^{f_n}&K_n\ar@{-->}[r]^{\delta_n}&}$ in $\xi$ and $d_n=g_{n-1}f_n$.
}\end{definition}

\begin{definition} \cite[Definition 4.5]{HZZ}
{\rm Let $\mathcal{W}$ be a class of objects in $\mathcal{C}$. An $\mathbb{E}$-triangle
$$\xymatrix@C=2em{A\ar[r]& B\ar[r]& C\ar@{-->}[r]& }$$ in $\xi$ is called to be
{\it $\mathcal{C}(-,\mathcal{W})$-exact} (respectively
{\it $\mathcal{C}(\mathcal{W},-)$-exact}) if for any $W\in\mathcal{W}$, the induced sequence of abelian group $\xymatrix@C=2em{0\ar[r]&\mathcal{C}(C,W)\ar[r]&\mathcal{C}(B,W)\ar[r]&\mathcal{C}(A,W)\ar[r]& 0}$ (respectively \\ $\xymatrix@C=2em{0\ar[r]&\mathcal{C}(W,A)\ar[r]&\mathcal{C}(W,B)\ar[r]&\mathcal{C}(W,C)\ar[r]&0}$) is exact in ${\rm Ab}$}.
\end{definition}

\begin{definition}\cite[Definition 4.6]{HZZ}
 {\rm Let $\mathcal{W}$ be a class of objects in $\mathcal{C}$. A complex $\mathbf{X}$ is called {\it $\mathcal{C}(-,\mathcal{W})$-exact} (respectively
{\it $\mathcal{C}(\mathcal{W},-)$-exact}) if it is an $\xi$-exact complex
$$\xymatrix@C=2em{\cdots\ar[r]&X_1\ar[r]^{d_1}&X_0\ar[r]^{d_0}&X_{-1}\ar[r]&\cdots}$$ in $\mathcal{C}$ such that  there is  $\mathcal{C}(-,\mathcal{W})$-exact (respectively
 $\mathcal{C}(\mathcal{W},-)$-exact) $\mathbb{E}$-triangle $$\xymatrix@C=2em{K_{n+1}\ar[r]^{g_n}&X_n\ar[r]^{f_n}&K_n\ar@{-->}[r]^{\delta_n}&}$$ in $\xi$ for each integer $n$ and $d_n=g_{n-1}f_n$.

 An $\xi$-exact complex $\mathbf{X}$ is called {\it complete $\mathcal{P}(\xi)$-exact} (respectively {\it complete $\mathcal{I}(\xi)$-exact}) if it is $\mathcal{C}(-,\mathcal{P}(\xi))$-exact (respectively
 $\mathcal{C}(\mathcal{I}(\xi),-)$-exact).}
\end{definition}

\begin{definition}\cite[Definition 4.7]{HZZ}
 {\rm A  {\it complete $\xi$-projective resolution}  is a complete $\mathcal{P}(\xi)$-exact complex\\ $$\xymatrix@C=2em{\mathbf{P}:\cdots\ar[r]&P_1\ar[r]^{d_1}&P_0\ar[r]^{d_0}&P_{-1}\ar[r]&\cdots}$$ in $\mathcal{C}$ such that $P_n$ is $\xi$-projective for each integer $n$. Dually,   a  {\it complete $\xi$-injective coresolution}  is a complete $\mathcal{I}(\xi)$-exact complex $$\xymatrix@C=2em{\mathbf{I}:\cdots\ar[r]&I_1\ar[r]^{d_1}&I_0\ar[r]^{d_0}&I_{-1}\ar[r]&\cdots}$$ in $\mathcal{C}$ such that $I_n$ is $\xi$-injective for each integer $n$.}
\end{definition}

\begin{definition} \cite[Definition 4.8]{HZZ}
{\rm  Let $\mathbf{P}$ be a complete $\xi$-projective resolution in $\mathcal{C}$. So for each integer $n$, there exists a $\mathcal{C}(-, \mathcal{P}(\xi))$-exact $\mathbb{E}$-triangle $\xymatrix@C=2em{K_{n+1}\ar[r]^{g_n}&P_n\ar[r]^{f_n}&K_n\ar@{-->}[r]^{\delta_n}&}$ in $\xi$. The objects $K_n$ are called {\it $\xi$-$\mathcal{G}$projective} for each integer $n$. Dually if  $\mathbf{I}$ is a complete $\xi$-injective  coresolution in $\mathcal{C}$, there exists a  $\mathcal{C}(\mathcal{I}(\xi), -)$-exact $\mathbb{E}$-triangle $\xymatrix@C=2em{K_{n+1}\ar[r]^{g_n}&I_n\ar[r]^{f_n}&K_n\ar@{-->}[r]^{\delta_n}&}$ in $\xi$ for each integer $n$. The objects $K_n$ are called {\it $\xi$-$\mathcal{G}$injective} for each integer $n$.}
\end{definition}

Similar to the way of defining $\mathcal{E}$-projective and $\mathcal{E}$-injective dimensions, for an object $A\in{\mathcal{C}}$, the $\mathcal{E}$-Gprojective dimension $\mathcal{E}$-${\rm \mathcal{G}pd} A$ and $\mathcal{E}$-Ginjective dimension $\mathcal{E}$-${\rm \mathcal{G}id} A$ are defined inductively in \cite{HZZ}.

Throughout, the full subcategory of $\xi$-projective (respectively $\xi$-injective) objects is denoted by $\mathcal{P(\xi)}$ (respectively $\mathcal{I(\xi)}$). We denote by $\mathcal{GP}(\xi)$ (respectively $\mathcal{GI}(\xi)$) the class of $\xi$-$\mathcal{G}$projective (respectively $\xi$-$\mathcal{G}$injective) objects.
It is obvious that $\mathcal{P(\xi)}$ $\subseteq$ $\mathcal{GP}(\xi)$ and $\mathcal{I(\xi)}$ $\subseteq$ $\mathcal{GI}(\xi)$.

\section{\bf Derived functors and Gorenstein homological dimensions for extriangulated categories}
 Throughout this section, we always assume that  $\mathcal{C}=(\mathcal{C}, \mathbb{E}, \mathfrak{s})$ is an extriangulated category with enough $\xi$-projectives and enough $\xi$-injectives satisfying weak idempotent completeness (see \cite[Condition 5.8]{NP}) and $\xi$ is a proper class of $\mathbb{E}$-triangles in  $(\mathcal{C}, \mathbb{E}, \mathfrak{s})$.

\begin{definition}\label{df:resolution} Let $A$ be an object in $\mathcal{C}$. An {\it $\xi$-projective} resolution of $A$ is an $\xi$-exact complex $\mathbf{P}\rightarrow A$ such that $P_n\in{\mathcal{P}(\xi)}$ for all $n\geq0$. Dually, an {\it $\xi$-injective} coresolution of $A$ is an $\xi$-exact complex $ A\rightarrow \mathbf{I}$ such that $I^n\in{\mathcal{I}(\xi)}$ for all $n\geq0$.
\end{definition}

Using standard arguments
from relative homological algebra, one can prove a version of the comparison
theorem for $\xi$-projective resolutions (resp. $\xi$-injective coresolutions). It follows that any two  $\xi$-projective resolutions (resp. $\xi$-injective coresolutions) of an object $A$ are homotopy equivalent. So
we have the following definition.

\begin{definition} \label{df:derived-functors} Let $A$ and $B$ be objects in $\mathcal{C}$.

\begin{enumerate}
\item[{\rm (1)}] If we choose an $\xi$-projective resolution $\xymatrix{\mathbf{P}\ar[r]& A}$ of  $A$, then for any integer $n\geqslant 0$, the \emph{$\xi$-cohomology groups} $\xi{\rm xt}_{\mathcal{P}(\xi)}^n(A,B)$ are defined as
$$\xi{\rm xt}_{\mathcal{P}(\xi)}^n(A,B)=H^n({\mathcal{C}}(\mathbf{P},B)).$$

\item[{\rm (2)}] If we choose an
$\xi$-injective coresolution $\xymatrix{B\ar[r]&\mathbf{I}}$ of  $B$, then for any integer $n\geqslant 0$, the \emph{$\xi$-cohomology groups} $\xi{\rm xt}_{\IE}^n(A,B)$ are defined as $$\xi{\rm xt}_{\IE}^n(A,B)=H^n({\mathcal{C}}(A, \mathbf{I})).$$
\end{enumerate}
\end{definition}
In fact, with the modifications of the usual proof, one obtains the isomorphism $\xi{\rm xt}_{\mathcal{P}(\xi)}^n(A,B)\cong \xi{\rm xt}_{\IE}^n(A,B),$
which is denoted by $\xi{\rm xt}_{\xi}^n(A,B).$

It is easy to see that the following lemma holds by \cite[Lemma 4.14]{HZZ}.

\begin{lem}\label{lem:horseshoe lemma}{\rm (Horseshoe Lemma)} If  $\xymatrix{A\ar[r]^x&B\ar[r]^{y}&C\ar@{-->}[r]^{\delta}&}$ is an $\mathbb{E}$-triangle  in $\xi$, then we have the following commutative diagram
$$\xymatrix{\mathbf{P}_{A}\ar[r]\ar[d]&\mathbf{P}_{B}\ar[r]\ar[d] &\mathbf{P}_{C}\ar[d]\\
A\ar[r]&B\ar[r]&C,}$$
where $\xymatrix{\mathbf{P}_{A}\ar[r]& A,}$ $\xymatrix{\mathbf{P}_{B}\ar[r]& B}$ and $\xymatrix{\mathbf{P}_{C}\ar[r]& C}$ are $\xi$-projective resolutions for $A$, $B$ and $C$ respectively.
\end{lem}
Using classical methods in homological algebra, one can see that for any $\mathbb{E}$-triangle in $\xi$, the long exact sequence of ``$\xi{\rm xt}$" functors exists. More precisely, we have the following lemma.

\begin{lem}\label{lem:long-exact-sequence} If  $\xymatrix{A\ar[r]^x&B\ar[r]^{y}&C\ar@{-->}[r]^{\delta}&}$ is an $\mathbb{E}$-triangle  in $\xi$, then for any object $X$ in $\mathcal{C}$, we have the following long exact sequences in ${\rm Ab}$
$$\xymatrix{0\ar[r]&{\xi{\rm xt}}_{\xi}^0(X,A)\ar[r]&{\xi{\rm xt}}_{\xi}^0(X,B)\ar[r]&{\xi{\rm xt}}_{\xi}^0(X,C)\ar[r]&{\xi{\rm xt}}_{\xi}^{1}(X,A)\ar[r]&\cdots}$$
and
$$\xymatrix{0\ar[r]&{\xi{\rm xt}}_{\xi}^0(C,X)\ar[r]&{\xi{\rm xt}}_{\xi}^0(B,X)\ar[r]&{\xi{\rm xt}}_{\xi}^0(A,X)\ar[r]&{\xi{\rm xt}}_{\xi}^{1}(C,X)\ar[r]&\cdots.}$$
\end{lem}

For any objects $A$ and $B$, there is always a natural map $$\xymatrix@C=1.5em{\delta: {\mathcal{C}}(A,B)\ar[r]& \xi {\rm xt}_{\xi}^0(A,B),}$$
which is an isomorphism if $A\in{\mathcal{P}(\xi)}$ or $B\in{\mathcal{I}(\xi)}$.

\begin{lem}\label{lem:2.4} Assume that $M$ is an object in $\mathcal{C}$ and $G$ is an object in $\mathcal{GP}(\xi)$. If $M\in{\widetilde{\mathcal{P}}(\xi)}$ or $M\in{\widetilde{\mathcal{I}}(\xi)}$, then $\xi{\rm xt}_{\xi}^0(G,M)\cong\mathcal{C}(G,M)$ and $\xi{\rm xt}_{\xi}^i(G,M)=0$ for any $i>0$.
\end{lem}
\begin{proof} Let $\xymatrix{G'\ar[r]^x&P\ar[r]^{y}&G\ar@{-->}[r]^{\delta}&}$ be an $\mathbb{E}$-triangle  in $\xi$ with $G'\in{\mathcal{GP}(\xi)}$ and $P\in\mathcal{P}(\xi)$. If $M\in{\widetilde{\mathcal{P}}(\xi)}$, then we have the following commutative diagram
$$\xymatrix@C=2em{0\ar[r]&{\mathcal{C}}(G,M)\ar[r]\ar[d]_{\delta_1}&{\mathcal{C}}(P,M)\ar[r]\ar[d]_{\delta_2}^{\cong}&{\mathcal{C}}(G',M)
\ar[d]^{\delta_3}\ar[r]&0&\\
0\ar[r]&{\rm{\xi}xt}_{\xi}^0(G,M)\ar[r]&{\rm {\xi}xt}_{\xi}^0(P,M)\ar[r]&{\rm {\xi}xt}_{\xi}^0(G',M)\ar[r]&\ar[r]{\rm {\xi}xt}_{\xi}^1(G,M)&0.}$$
where the top exact sequence follows from \cite[Lemma 4.10(2)]{HZZ}. It is easy to see that $\delta_1$ is monic, similarly one can prove that $\delta_3$ is monic, hence $\delta_1$ is an epimorphism by snake lemma, so $\delta_1$ is an isomorphism. Similarly one can get that $\delta_3$ is an isomorphism, so $\xi{\rm xt}_{\xi}^1(G,M)=0$.  It is easy to show that $\xi{\rm xt}_{\xi}^i(G,M)=0$ for any $i>0$ by Lemma \ref{lem:long-exact-sequence}.

If $M\in{\widetilde{\mathcal{I}}(\xi)}$, a similar argument yields that $\xi{\rm xt}_{\xi}^0(G,M)\cong\mathcal{C}(G,M)$ and $\xi{\rm xt}_{\xi}^i(G,M)=0$ for any $i>0$.
\end{proof}

\begin{lem}\label{lem:2.6} Let $\xymatrix@C=2em{A\ar[r]^x& B\ar[r]^y& C\ar@{-->}[r]^\rho&}$ be an $\mathbb{E}$-triangle in $\xi$ with $A,B\in\mathcal{GP}(\xi)$, then  $C\in\mathcal{GP}(\xi)$ if and only if $\xi{\rm xt}_{\xi}^1(C,P)=0$ for any $P\in{\mathcal{P}(\xi)}$.
\end{lem}
\begin{proof} The ``only if''  part follows from  \cite[Lemma 4.10(2)]{HZZ} and Lemma \ref{lem:2.4}.

For the ``if'' part, since $A\in\mathcal{GP}(\xi)$, there exists an $\mathbb{E}$-triangle $\xymatrix@C=1.5em{A\ar[r]&P\ar[r]&K\ar@{-->}[r]&}$ in $\xi$ with $P\in \mathcal{P}(\xi)$ and $K\in \mathcal{GP}(\xi)$.
Then we have the following commutative diagram:
$$\xymatrix{A\ar[r]\ar[d]&B\ar[r]\ar[d]&C\ar@{=}[d]\ar@{-->}[r]&\\
 P\ar[r]\ar[d]&G\ar[r]\ar[d]&C\ar@{-->}[r]&\\
 K\ar@{-->}[d]\ar@{=}[r]&K\ar@{-->}[d]&&\\
 &&&}$$
where all rows and columns are $\mathbb{E}$-triangles in $\xi$ because $\xi$ is closed under cobase change. It follows from Theorem \cite[Theorem 4.16]{HZZ} that $G\in \mathcal{P}(\xi)$. For the $\mathbb{E}$-triangle  $\xymatrix@C=1.5em{P\ar[r]&G\ar[r]&C\ar@{-->}[r]&}$, we have the following commutative diagram
$$\xymatrix@C=2em{&{\mathcal{C}}(C,P)\ar[r]\ar[d]_{\delta_1}&{\mathcal{C}}(G,P)\ar[r]\ar[d]_{\delta_2}^{\cong}&{\mathcal{C}}(P,P)
\ar[d]_{\delta_3}^{\cong}\ar@{-->}[r]&0&\\
0\ar[r]&{\rm{\xi}xt}_{\xi}^0(C,P)\ar[r]&{\rm {\xi}xt}_{\xi}^0(G,P)\ar[r]&{\rm {\xi}xt}_{\xi}^0(P,P)\ar[r]&0.}$$
by hypothesis. Hence the $\mathbb{E}$-triangle  $\xymatrix@C=1.5em{P\ar[r]&G\ar[r]&C\ar@{-->}[r]&}$ is split, and $C\in\mathcal{GP}(\xi)$.
\end{proof}

\begin{lem}\cite[Proposition 5.6]{HZZ} \label{pro5} Let  $M$ be an object in $\widetilde{\mathcal{GP}}(\xi)$. Then $\xi$-$\mathcal{G}{\rm pd} M\leqslant n$ if and only if
for any $\xi$-exact complex $\xymatrix{K_n\ar[r]&G_{n-1}\ar[r]&\cdots\ar[r]&G_1\ar[r]&G_0\ar[r]&A,}$ $K_n$ belongs to $\widetilde{\mathcal{GP}}(\xi)$ provided that $G_i$ is $\xi$-$\mathcal{G}$projective for $0\leq i\leq n-1$.
\end{lem}

We are now in a position to prove the main result of this section.

\begin{thm}\label{prop:2.7} Let $M$ be an object in $\widetilde{\mathcal{GP}}(\xi)$. Then the following are equivalent:
 \begin{enumerate}
\item[{\rm (1)}] $\xi\textrm{-}\mathcal{G}{\rm pd}M\leq n$;
\item[{\rm (2)}] $\xi{\rm xt}_{\xi}^{i}(M,Q)=0$ for any $Q\in{\widetilde{\mathcal{P}}(\xi)}$ and $i\geq {n+1}$;
\item[{\rm (3)}] $\xi{\rm xt}_{\xi}^{i}(M,Q)=0$ for any $Q\in{\mathcal{P}(\xi)}$ and $i\geq {n+1}$.
\end{enumerate}
\noindent Furthermore, we have the following equalities:
\vspace{2mm}
\begin{center}
{$\xi\textrm{-}\mathcal{G}{\rm pd}M={\rm sup}\{ i\in{\mathbb{N}_0} \ | \ \xi{\rm xt}_{\xi}^{i}(M,Q)(M,Q) \neq 0 \ \textrm{for} \ \textrm{some} \ Q\in{\mathcal{P}(\xi)}\}$

\vspace{2mm}
\hspace{14mm} $={\rm sup}\{ i\in{\mathbb{N}_0} \ | \ \xi{\rm xt}_{\xi}^{i}(M,Q)(M,Q) \neq 0 \ \textrm{for} \ \textrm{some} \ Q\in{\widetilde{\mathcal{P}}(\xi)}\}.$
}
\end{center}
\end{thm}
\begin{proof} $(1)\Rightarrow (2)$. Let
$\xymatrix@C=2em{\cdots\ar[r]&P_n\ar[r]&\cdots \ar[r]&P_1\ar[r]&P_0\ar[r]&M}$ be an $\xi$-projective resolution of $M$. Then there exists an $\mathbb{E}$-triangle $\xymatrix@C=2em{K_{n+1}\ar[r]&P_n\ar[r]&K_n\ar@{-->}[r]&}$ in $\xi$ with $K_0=M$ for any integer $n\geq0$. Thus $K_n$ belongs to $\widetilde{\mathcal{GP}}(\xi)$ by Lemma \ref{pro5}. Note that $\xi{\rm xt}_{\xi}^{n+j}(M,Q)\cong \xi{\rm xt}_{\xi}^{j}(K_n,Q)$ for any $Q\in{\widetilde{\mathcal{P}}(\xi)}$ and $j\geq {1}$. It follows from Lemma \ref{lem:2.4} that $\xi{\rm xt}_{\xi}^{j}(K_n,Q)=0$ for any $j\geq {1}$. So $\xi{\rm xt}_{\xi}^{i}(M,Q)=0$ for any $Q\in{\widetilde{\mathcal{P}}(\xi)}$ and $i\geq {n+1}$.

$(2)\Rightarrow (3)$ is trivial.

$(3)\Rightarrow (1)$.  By hypothesis, we assume that $\xi\textrm{-}\mathcal{G}{\rm pd}M\leq m$. If $m\leq n$, then (1) holds. If $m> n$, we consider the following $\xi$-projective resolution of $M$
$$\xymatrix@C=2em{\cdots\ar[r]&P_n\ar[r]&\cdots \ar[r]&P_1\ar[r]&P_0\ar[r]&M.}$$
Thus there exists an $\mathbb{E}$-triangle
$$\xymatrix@C=2em{K_{i+1}\ar[r]&P_i\ar[r]&K_i\ar@{-->}[r]&}$$ in $\xi$ with $K_0=M$ for any integer $i\geq0$. Since $\xi{\rm xt}_{\xi}^{m}(M,Q)\cong \xi{\rm xt}_{\xi}^{1}(K_{m-1},Q)$ for any $Q\in{{\mathcal{P}}(\xi)}$, $\xi{\rm xt}_{\xi}^{1}(K_{m-1},Q)=0$ by (3). Note that $\xymatrix@C=2em{K_{m}\ar[r]&P_{m-1}\ar[r]&K_{m-1}\ar@{-->}[r]&}$ is an $\mathbb{E}$-triangle in $\xi$ with $K_{m}$ and $P_{m-1}$ are $\xi$-$\mathcal{G}$projective by Lemma \ref{pro5}. It follows from Lemma \ref{lem:2.6} that $K_{m-1}$ is $\xi$-$\mathcal{G}$projective. Thus we get that $\xi\textrm{-}\mathcal{G}{\rm pd}M\leq m-1$. By proceeding in this manner, we get that $m\leq n$. So $\xi\textrm{-}\mathcal{G}{\rm pd}M\leq n$.

The last formulas in the theorem for determination of $\xi\textrm{-}\mathcal{G}{\rm pd}M$ are a direct consequence
of the equivalence between (1)-(3).
\end{proof}

\section{\bf Global Gorenstein homological dimensions for extriangulated categories}

We begin this section with the following concept which is inspired from (co)generating subcategories of triangulated categories.

\begin{definition}\label{df:generated-subcategory}
Let $\mathcal{X}$ be a full subcategory of $\mathcal{C}$. Then $\mathcal{X}$ is called a \emph{generating subcategory} of $\mathcal{C}$ if for
all $A\in{\mathcal{C}}$, $\mathcal{C}(\mathcal{X},A)=0$ implies that $A=0$. Dually, a full subcategory of $\mathcal{Y}$ of $\mathcal{C}$ is called a \emph{cogenerating subcategory} of $\mathcal{C}$ if for
all $B\in{\mathcal{C}}$, $\mathcal{C}(B,\mathcal{X})=0$ implies that $B=0$.
\end{definition}

Next we assign two invariants to an extriangulated category, which is motivated by Gedrich and Gruenberg's invariants of a ring \cite{GG}, and Asadollahi and Salarian's invariants to a triangulated category \cite{AS2}.
\begin{definition}\label{df:silp-spli}
Let $\mathcal{C}$ be an extriangulated category. We assign two invariants to $\mathcal{C}$ as follows:
{\rm $$\xi\textrm{-}{\rm silp}\mathcal{C}=\sup\{\xi\textrm{-}{\rm id}P \ | \ P\in{\mathcal{P}(\xi)}\},$$
$$\xi\textrm{-}{\rm spli}\mathcal{C}=\sup\{\xi\textrm{-}{\rm pd}I \ | \ I\in{\mathcal{I}(\xi)}\}.$$}
\end{definition}

\begin{prop}\label{prop:3.1}If for any extriangulated category $\mathcal{C}$ with both $\xi\textrm{-}{\rm silp}\mathcal{C}$ and $\xi\textrm{-}{\rm spli}\mathcal{C}$
are finite, then they are equal.
\end{prop}
\begin{proof} The proof is similar to \cite[Proposition 4.2]{AS2}.
\end{proof}

The following lemma is essentially taken from \cite[Propositions 3.7 and 3.8]{RL1}. However, it can be easily
extended to our setting by noting that Theorem \ref{prop:2.7} holds.
\begin{lem}\label{prop:3.2}Let $\mathcal{C}$ be an extriangulated category.
\begin{enumerate}
\item[{\rm (1)}] If $\sup\{\xi\textrm{-}\mathcal{G}{\rm id}M \ | \ \textrm{for} \ \textrm{any} \ M\in{\mathcal{C}}\}\leq \infty$ and $\mathcal{P}(\xi)$ is a generating subcategory of $\mathcal{C}$,
then $\xi\textrm{-}{\rm spli}\mathcal{C}<\infty$.

\item[{\rm (2)}] If  $\sup\{\xi\textrm{-}\mathcal{G}{\rm pd}M \ | \ \textrm{for} \ \textrm{any} \ M\in{\mathcal{C}}\}\leq \infty$ and $\mathcal{I}(\xi)$ is a cogenerating subcategory of $\mathcal{C}$,
then $\xi\textrm{-}{\rm silp}\mathcal{C}<\infty$.
\end{enumerate}
\end{lem}

\begin{lem}\label{lem:2.3} If $\xymatrix{A\ar[r]^x&B\ar[r]^{y}&C\ar@{-->}[r]^{\delta}&}$ is an $\mathbb{E}$-triangle  in $\xi$ with $C\in{\mathcal{GP}(\xi)}$, then it is  $\mathcal{C}(-,\widetilde{\mathcal{I}}(\xi))$-exact.
\end{lem}

\begin{proof} The proof is similar to \cite[Lemma 5.7]{HZZ}, we give its proof for convenience. Since $C\in\mathcal{GP}(\xi)$, there is an $\mathbb{E}$-triangle $\xymatrix{K\ar[r]^g&P\ar[r]^f&C\ar@{-->}[r]^\delta&}$ in $\xi$ with $P\in\mathcal{P}(\xi)$ and $K\in\mathcal{GP}(\xi)$. It suffices to show that  the $\mathbb{E}$-triangle $\xymatrix{K\ar[r]^g&P\ar[r]^f&C\ar@{-->}[r]^\delta&}$ is $\mathcal{C}(-,\widetilde{\mathcal{I}}(\xi))$-exact by \cite[Lemma 4.10(1)]{HZZ} because  any $\mathbb{E}$-triangle in $\xi$ which end in $P$ is split, that is, the sequence of abelian groups $$\xymatrix{(*)&0\ar[r]&\mathcal{C}(C, M)\ar[r]^{\mathcal{C}(f, M)}&\mathcal{C}(P, M)\ar[r]^{\mathcal{C}(g, M)}&\mathcal{C}(K, M)\ar[r]&0}$$ is exact for any $M\in\mathcal{C}$ with $\xi$-${\rm id}M=n$. If $n=0$, then the sequence $(*)$ is exact because $M\in\mathcal{I}(\xi)$.

Assume that the sequence $(*)$ is exact for any $L\in\mathcal{C}$ with $\xi$-${\rm id}L=n-1$. Now we consider the case of $\xi$-${\rm id}M=n$. Suppose that $\xymatrix{M\ar[r]^x&I\ar[r]^y&L\ar@{-->}[r]^\rho&}$ is an $\mathbb{E}$-triangle in $\xi$ with $I\in\mathcal{I}(\xi)$ and $\xi$-${\rm id}L=n-1$. Hence the sequence of abelian groups $$\xymatrix{0\ar[r]&\mathcal{C}(C, L)\ar[r]^{\mathcal{C}(f, L)}&\mathcal{C}(P, L)\ar[r]^{\mathcal{C}(g, L)}&\mathcal{C}(K, L)\ar[r]&0}$$ is exact by induction. Since  the functor $\mathcal{C}(-, -)$ is biadditive functor, we have following commutative diagram
$$\xymatrix{&0\ar@{-->}[d]&0\ar[d]&0\ar[d]\\
0\ar@{-->}[r]&\mathcal{C}(C, M)\ar[r]^{\mathcal{C}(C, x)}\ar[d]_{\mathcal{C}(f, M)}&\mathcal{C}(C, I)\ar[d]^{\mathcal{C}(f, I)}\ar[r]^{\mathcal{C}(C, y)}&\mathcal{C}(C, L)\ar[d]^{\mathcal{C}(f, L)}\ar@{-->}[r]&0\\
0\ar[r]&\mathcal{C}(P, M)\ar[r]^{\mathcal{C}(P, x)}\ar[d]_{\mathcal{C}(g, M)}&\mathcal{C}(P, I)\ar[d]^{\mathcal{C}(g, I)}\ar[r]^{\mathcal{C}(P, y)}&\mathcal{C}(P, L)\ar[d]^{\mathcal{C}(g, L)}\ar[r]&0\\
0\ar@{-->}[r]&\mathcal{C}(K, M)\ar[r]^{\mathcal{C}(K, x)}\ar@{-->}[d]&\mathcal{C}(K,I)\ar[r]^{\mathcal{C}(K, y)}\ar[d]&\mathcal{C}(K,L)\ar@{-->}[r]\ar[d]&0\\
&0&0&0}$$
It is easy to see that the morphism $\mathcal{C}(K, y)$ is epic. Similarly $\mathcal{C}(C, y)$ is epic. Next we claim that $\mathcal{C}(g,M)$ is epic. For any morphism $h: K\rightarrow M$, there exists a morphism $h_1: P\rightarrow I$ such that $h_1g=xh$. It follows from (ET3) that there exists a morphism $h_2$ which gives the following commutative diagram:
 $$\xymatrix{K\ar[r]^g\ar[d]^{h}&P\ar[r]^f\ar[d]^{h_1}&C\ar@{-->}[r]^\delta\ar@{-->}[d]^{h_2}&\\
 M\ar[r]^x&I\ar[r]^y&L\ar@{-->}[r]^{\rho}&}$$
 Since $\mathcal{C}(C, y): \mathcal{C}(C, I)\rightarrow \mathcal{C}(C, L)$ is epic,  $h_2$ factors through $y$. Hence $h$ factors through $g$ by \cite[Corollary 3.5]{NP}. That is, $\mathcal{C}(g, M)$ is epic, which implies
that the sequence of abelian groups $$\xymatrix{0\ar[r]&\mathcal{C}(K, M)\ar[r]^{\mathcal{C}(K, x)}&\mathcal{C}(K, I)\ar[r]^{\mathcal{C}(K, y)}&\mathcal{C}(K, L)\ar[r]&0}$$  is exact by snake lemma. Similarly one can prove that the sequence of abelian groups $$\xymatrix{0\ar[r]&\mathcal{C}(C, M)\ar[r]^{\mathcal{C}(C, x)}&\mathcal{C}(C, I)\ar[r]^{\mathcal{C}(C, y)}&\mathcal{C}(C, L)\ar[r]&0}$$ is exact as $C\in\mathcal{GP}(\xi)$. It is straightforward to prove that the
 sequence $(*)$ is exact by  Lemma $3\times 3$.
\end{proof}

\begin{prop}\label{thm:2.5} Let $\mathcal{C}$ be an extriangulated category.
 \begin{enumerate}
\item[{\rm (1)}] If $M$ is an object in ${\widetilde{\mathcal{I}}(\xi)}$, then $\xi\textrm{-}\mathcal{G}{\rm pd}M =\xi\textrm{-}{\rm pd}M$.
\item[{\rm (2)}] If $M$ is an object in ${\widetilde{\mathcal{P}}(\xi)}$, then $\xi\textrm{-}\mathcal{G}{\rm id}M =\xi\textrm{-}{\rm id}M$.
\end{enumerate}
\end{prop}
\begin{proof} We only prove (1), the proof of (2) is similar. It is clear that $\xi\textrm{-}\mathcal{G}{\rm pd}M \leq\xi\textrm{-}{\rm pd}M$, and the equality holds if $\xi\textrm{-}\mathcal{G}{\rm pd}M =\infty$. If $\xi\textrm{-}\mathcal{G}{\rm pd}M =n<\infty$, then there exists an $\mathbb{E}$-triangle $\xymatrix@C=1.5em{M\ar[r]&L\ar[r]&G\ar@{-->}[r]&}$ in $\xi$ such that $G\in\mathcal{GP}(\xi)$ and $\xi$-${\rm pd} L\leqslant n$ by \cite[Proposition 5.9]{HZZ}. Note that $M\in{\widetilde{\mathcal{I}}(\xi)}$. Thus, by Lemma \ref{lem:2.3}, we have the following exact sequence $$\xymatrix{0\ar[r]&\mathcal{C}(G, M)\ar[r]&\mathcal{C}(L, M)\ar[r]&\mathcal{C}(M, M)\ar[r]&0.}$$ So $M$ is a direct summand of $L$, and $\xi\textrm{-}{\rm pd}M \leq\xi\textrm{-}{\rm pd}L=n=\xi\textrm{-}\mathcal{G}{\rm pd}M$, as desired.
\end{proof}

We are now in a position to prove the main result of this section.

\begin{thm}\label{thm:3.4}Let $\mathcal{C}$ be an extriangulated category, and let $\mathcal{P}(\xi)$ be a generating subcategory of $\mathcal{C}$ and $\mathcal{I}(\xi)$ be a cogenerating subcategory of $\mathcal{C}$. Consider the following
conditions for any non-negative integer $m$:
\begin{enumerate}
\item[{\rm (1)}] $\sup\{\xi\textrm{-}\mathcal{G}{\rm pd}M \ | \ \textrm{for} \ \textrm{any} \ M\in{\mathcal{C}}\}\leq m$.

\item[{\rm (2)}] $\sup\{\xi\textrm{-}\mathcal{G}{\rm id}M \ | \ \textrm{for} \ \textrm{any} \ M\in{\mathcal{C}}\}\leq m$.

\item[{\rm (3)}]  $\xi\textrm{-}{\rm spli}\mathcal{C}=\xi\textrm{-}{\rm silp}\mathcal{C}\leq m$.
\end{enumerate}
Then ${\rm (1)}\Rightarrow {\rm (3)}$ and ${\rm (2)}\Rightarrow {\rm (3)}$. The converses hold if $\mathcal{C}$ satisfies the following condition:

{\rm \textbf{Condition $(\star)$}}: If  $N\in\mathcal{C}$ and $M\in{\widetilde{\mathcal{P}}(\xi)}$  such that $\xi{\rm xt}_{\xi}^{i}(M,N)=0$ for  any $i\geq1$, then $\mathcal{C}(M,N)\cong\xi{\rm xt}_{\xi}^{0}(M,N)$. Dually, if $N\in\mathcal{C}$  and $M\in{\widetilde{\mathcal{I}}(\xi)}$  such that $\xi{\rm xt}_{\xi}^{i}(N,M)=0$ for any  $i\geq1$, then $\mathcal{C}(N,M)\cong\xi{\rm xt}_{\xi}^{0}(N,M)$.


\end{thm}
\begin{proof} $(1)\Rightarrow(3)$ and $(2)\Rightarrow(3)$  hold by Proposition \ref{prop:3.1}, Proposition \ref{thm:2.5} and Lemma \ref{prop:3.2}.

$(3)\Rightarrow(1)$. Assume that $\xi\textrm{-}{\rm spli}\mathcal{C}=\xi\textrm{-}{\rm silp}\mathcal{C}=m<\infty$. For any $M\in{\mathcal{C}}$, we consider the following $\xi$-projective resolution of $M$
$$\xymatrix@C=2em{\cdots\ar[r]&P_n\ar[r]&\cdots \ar[r]&P_1\ar[r]&P_0\ar[r]&M.}$$
Thus there exists an $\mathbb{E}$-triangle
$$\xymatrix@C=2em{K_{n+1}\ar[r]&P_n\ar[r]&K_n\ar@{-->}[r]&}$$ in $\xi$ with $K_0=M$ for any integer $n\geq0$. Consider the $\xi$-injective coresolution of $K_{i}$
$$\xymatrix@C=2em{K_{i}\ar[r]&I_{i}^{0}\ar[r]&I_{i}^{1}\ar[r]&\cdots}$$
for any integer $i\geq0$. Thus we have the following commutative diagram which is the dual of Lemma \ref{lem:horseshoe lemma}
 $$
\xymatrix{
  K_{n+1} \ar[d]\ar[r] & P_{n} \ar[d]\ar[r] & K_{n}\ar[d] \ar[r]  & \\
  I_{n+1}^{0} \ar[d]\ar[r] & I_{n+1}^{0}\oplus I_{n}^{0} \ar[d]\ar[r] & I_{n}^{0}\ar[d] \ar[r]  & \\
  I_{n+1}^{1} \ar[d]\ar[r] & I_{n+1}^{1}\oplus I_{n}^{1} \ar[d]\ar[r] & I_{n}^{1}\ar[d] \ar[r]  & \\
  \vdots \ar[d]& \vdots \ar[d] &\vdots\ar[d] &\\
   I_{n+1}^{m-1} \ar[d]\ar[r] & I_{n+1}^{m-1}\oplus I_{n}^{m-1} \ar[d]\ar[r] & I_{n}^{m-1}\ar[d] \ar[r]  & \\
   L_{n+1}^{m} \ar[r] & I_{n}^{m}\ar[r] & L_{n}^{m} \ar[r]&. }
  $$
  Note that $\xi\textrm{-}\textrm{id}P_{n}\leq m$ and $I_{n+1}^{0}\oplus I_{n}^{0},\ \cdots, \ I_{n+1}^{m-1}\oplus I_{n}^{m-1}\in{\mathcal{I}(\xi)}$. It follows that
  $I_{n}^{m}\in{\mathcal{I}(\xi)}$.
  For any $I\in{\mathcal{I}(\xi)}$,  since $\xi\textrm{-}\textrm{pd}I\leq m$ by hypothesis, we get that $\xi{\rm xt}_{\xi}^{s}(I,L_{j}^{m})\cong \xi{\rm xt}_{\xi}^{m+s}(I,K_{j})=0$ for any integers $j\geq0$ and $s\geq 1$.
\medskip

Note that $I\in\widetilde{\mathcal{P}}(\xi)$ and $\xi{\rm xt}_{\xi}^{s}(I,L_{j}^{m})=0$ for any $j\geq0$ and $s\geq1$ by the above proof. It follows from the hypothesis that $\mathcal{C}(I,L_{j}^{m})\cong\xi{\rm xt}_{\xi}^{0}(I,L_{j}^{m})$ for any integer $j\geq0$. For the $\mathbb{E}$-triangle
$\xymatrix@C=2em{L_{j+1}^{m}\ar[r]&I_n^{m}\ar[r]&L_{j}^{m}\ar@{-->}[r]&}$ in $\xi$ with $j\geq0$, we have the following commutative diagram
$$\xymatrix@C=3em{0\ar@{-->}[r]&{\mathcal{C}}(I,L_{j+1}^{m})\ar[r]\ar[d]_{\cong}&{\mathcal{C}}(I,I_{j}^{m})\ar[r]\ar[d]^{\cong}&{\mathcal{C}}(I,L_{j}^{m})
\ar[d]^{\cong}\ar@{-->}[r]&0\\
0\ar[r]&{\rm{\xi}xt}_{\xi}^0(I,L_{j+1}^{m})\ar[r]&{\rm {\xi}xt}_{\xi}^0(I,I_{j+1}^{m})\ar[r]&{\rm {\xi}xt}_{\xi}^0(I,L_{j}^{m})\ar[r]&0.}$$
Therefore, there exists a $\xi$-exact complex $$\xymatrix{\cdots\ar[r]&I_{1}^{m}\ar[r]&I_{0}^{m}\ar[r]&L_{0}^{m},}$$
which is $\mathcal{C}(\mathcal{I}(\xi),-)$-exact.
Note that there exists an $\mathbb{E}$-triangle
$\xymatrix@C=2em{L_{0}^{m}\ar[r]&E^{0}\ar[r]&Y^1\ar@{-->}[r]&}$ in $\xi$ with $E^{0}\in{\mathcal{I}(\xi)}$. By the above proof, we have that $\xi{\rm xt}_{\xi}^{i}(I,L_{0}^{m})=0$ for any $i\geq1$. It follows from Lemma \ref{lem:long-exact-sequence} that $\xi{\rm xt}_{\xi}^{i}(I,Y^{1})=0$ for any $i\geq1$. Hence $\mathcal{C}(I,Y^{1})\cong\xi{\rm xt}_{\xi}^{0}(I,Y^{1})$ by hypothesis.
By the foregoing proof, the $\mathbb{E}$-triangle
$\xymatrix@C=2em{L_{0}^{m}\ar[r]&E^{0}\ar[r]&Y^1\ar@{-->}[r]&}$ is $\mathcal{C}(\mathcal{I}(\xi),-)$-exact. By proceeding in this manner, we get a $\xi$-exact complex $$\xymatrix{L_{0}^{m}\ar[r]& E^{0}\ar[r]&E^{1}\ar[r]&\cdots,}$$ which is $\mathcal{C}(\mathcal{I}(\xi),-)$-exact. So $L_{0}^{m}\in{\mathcal{GI}(\xi)}$ and $\xi\textrm{-}\mathcal{G}{\rm id}M\leq m$, as desired.

$(3)\Rightarrow(2)$. The proof is similar to that of $(3)\Rightarrow(1)$.
\end{proof}

As a consequence of Theorem \ref{thm:3.4}, we have the following corollary.
\begin{cor}\label{corollary:4.5} Let $\mathcal{C}$ be an extriangulated category satisfying Condition $(\star)$. If $\mathcal{P}(\xi)$ is a generating subcategory of $\mathcal{C}$ and $\mathcal{I}(\xi)$ is a cogenerating subcategory of $\mathcal{C}$, then we have the following equality£»$$\sup\{\xi\textrm{-}\mathcal{G}{\rm pd}M \ | \ \textrm{for} \ \textrm{any} \ M\in{\mathcal{C}}\}=\sup\{\xi\textrm{-}\mathcal{G}{\rm id}M \ | \ \textrm{for} \ \textrm{any} \ M\in{\mathcal{C}}\}.$$
\end{cor}

Let $\mathcal{C}=(\mathcal{C}, \mathbb{E}, \mathfrak{s})$ be an extriangulated category, and let $\mathcal{U}$, $\mathcal{V}$ $\subseteq$ $\mathcal{C}$ be a pair of full additive subcategories, closed
under isomorphisms and direct summands. Recall from \cite[Definitions 4.1 and 4.12]{NP} that the pair ($\mathcal{U}$, $\mathcal{V}$) is called a cotorsion
pair on $\mathcal{C}$ if it satisfies the following conditions:

(1) $\mathbb{E}(\mathcal{U}, \mathcal{V})=0$;

(2) For any $C \in{\mathcal{C}}$, there exists a conflation $V^{C}\rightarrow U^{C}\rightarrow C$ satisfying
$U^{C}\in{\mathcal{U}}$ and $V^{C}\in{\mathcal{V}}$;

(3) For any $C \in{\mathcal{C}}$ , there exists a conflation $C\rightarrow V_{C} \rightarrow U_{C}$ satisfying
$U_{C}\in{\mathcal{U}}$ and $V_{C}\in{\mathcal{V}}$.

Moreover, if $(\mathcal{X},\mathcal{Y})$ and $(\mathcal{Y},\mathcal{Z})$ are cotorsion pairs on $\mathcal{C}$, then the triple $(\mathcal{X},\mathcal{Y},\mathcal{Z})$ is called the \emph{cotorsion triple} on $\mathcal{C}$. It follows from \cite[Theorem 3.2]{HZZ} that: for any proper class $\xi$ of $ \mathbb{E}$-triangles, $(\mathcal{C}, \mathbb{E}_\xi, \mathfrak{s}_\xi)$ is an extriangulated category, where  $\mathbb{E}_\xi:=\mathbb{E}|_\xi$ and $\mathfrak{s}_\xi:=\mathfrak{s}|_{\mathbb{E}_\xi}$. Next we have the following corollary.

\begin{cor}\label{corollary:4.6}  Assume that $\mathcal{C}$ is an extriangulated category satisfying Condition $(\star)$ and
$$\mathcal{P}^{\leqslant n}(\xi)=
\{M\in\mathcal{C} \ | \ \xi\textrm{-}{\rm pd}M\leqslant n\}.$$ Then $\sup\{\xi\textrm{-}\mathcal{G}{\rm pd}M \ | \ \textrm{for} \ \textrm{any} \ M\in{\mathcal{C}}\}\leq n$ if and only if $(\mathcal{GP}(\xi), \mathcal{P}^{\leqslant n}(\xi), \mathcal{GI}(\xi))$ is a cotorsion triple.
\end{cor}
\begin{proof} The ``if" part is straightforward by noting that $(\mathcal{GP}(\xi), \mathcal{P}^{\leqslant n}(\xi))$  is a cotorsion pair in $(\mathcal{C}, \mathbb{E}_\xi, \mathfrak{s}_\xi)$. For the ``only if" part, we assume that $\sup\{\xi\textrm{-}\mathcal{G}{\rm pd}M \ | \ \textrm{for} \ \textrm{any} \ M\in{\mathcal{C}}\}\leq n$. It follows from  \cite[Theorem 5.4]{HZZ} that $(\mathcal{GP}(\xi), \mathcal{P}^{\leqslant n}(\xi))$ is a cotorsion pair in $(\mathcal{C}, \mathbb{E}_\xi, \mathfrak{s}_\xi)$. Note that $\sup\{\xi\textrm{-}\mathcal{G}{\rm id}M \ | \ \textrm{for} \ \textrm{any} \ M\in{\mathcal{C}}\}\leq n$ by Corollary \ref{corollary:4.5}.  Thus one can get that $(\mathcal{I}^{\leqslant n}(\xi), \mathcal{GI}(\xi))$ is a cotorsion pair in $(\mathcal{C}, \mathbb{E}_\xi, \mathfrak{s}_\xi)$, where $\mathcal{I}^{\leqslant n}(\xi)$=
$\{M\in\mathcal{C} \ | \ \xi\textrm{-}{\rm id}M\leqslant n\}$. Next we claim that $\mathcal{P}^{\leqslant n}(\xi)=\mathcal{I}^{\leqslant n}(\xi)$. Let $M$ be an object in $\mathcal{I}^{\leqslant n}(\xi)$. Then $\xi\textrm{-}\mathcal{G}{\rm pd}M =\xi\textrm{-}{\rm pd}M$ by Theorem \ref{thm:2.5}. Thus $\xi\textrm{-}{\rm pd}M\leqslant n$, and hence $M\in{\mathcal{P}^{\leqslant n}(\xi)}$. Similarly, we can prove that $\mathcal{P}^{\leqslant n}(\xi)\subseteq \mathcal{I}^{\leqslant n}(\xi)$. This completes the proof.
\end{proof}

\begin{Ex}\label{Ex:3.12}
\emph{(1)} Assume that $(\mathcal{C}, \mathbb{E}, \mathfrak{s})$ is an exact category and $\xi$ is a class of exact sequences which is closed under isomorphisms. One can check that Condition $(\star)$ in Theorem \ref{thm:3.4} is automatically satisfied.

\emph{(2)} If $\mathcal{C}$ is a triangulated category and the class $\xi$ of triangles is closed under
isomorphisms and suspension \emph{(see \cite[Section 2.2]{Bel1} and \cite[Remark 3.4(3)]{HZZ})}, then Condition $(\star)$ in Theorem \ref{thm:3.4} is also satisfied.
\end{Ex}
\begin{proof}We only need to prove (2). Let $\mathcal{P}^{\leqslant n}(\xi)=\{M\in\mathcal{C} \ | \ \xi$-${\rm pd}M\leqslant n\}$.
Assume that $N$ is an object in $\mathcal{C}$ such that $\xi{\rm xt}_{\xi}^{i}(M,N)=0$ for $M\in{\mathcal{P}^{\leqslant n}(\xi)}$ and any $i\geq1$.
We proceed by induction on $n$. If $n=0$, then it is easy to check that $\mathcal{C}(M,N)\cong\xi{\rm xt}_{\xi}^{0}(M,N)$ for all $M\in{{\mathcal{P}(\xi)}}$. We suppose that $\mathcal{C}(M,N)\cong\xi{\rm xt}_{\xi}^{0}(M,N)$ provided that $\xi{\rm xt}_{\xi}^{i}(M,N)=0$ for all $M\in{{\mathcal{P}^{\leqslant n-1}(\xi)}}$ and any $i\geq1$. Let $M$ be an object in $\mathcal{C}$ with $\xi$-${\rm pd}M \leq n$ and  $\xi{\rm xt}_{\xi}^{i}(M,N)=0$ for any $i\geq1$. Then there exists a triangle
$\xymatrix@C=2em{K\ar[r]^{f}&P\ar[r]^{g}&M\ar[r]^{h}&\Sigma K}$ in $\xi$ with $P\in{\mathcal{P}(\xi)}$ and $K\in{\mathcal{P}^{\leqslant n-1}(\xi)}$. Consider the following commutative diagram:
$$\xymatrix@C=3em{0\ar@{-->}[r]&{\mathcal{C}}(M,N)\ar[r]^{\mathcal{C}(g,N)}\ar[d]_{\delta_1}&{\mathcal{C}}(P,N)\ar[r]^{\mathcal{C}(f,N)}\ar[d]^{\delta_2}&{\mathcal{C}}(K,N)
\ar[d]^{\delta_3}\ar@{-->}[r]&0\\
0\ar[r]&{\rm{\xi}xt}_{\xi}^0(M,N)\ar[r]&{\rm {\xi}xt}_{\xi}^0(P,N)\ar[r]&{\rm {\xi}xt}_{\xi}^0(K,N)\ar[r]&0.}$$
Since both $\delta_2$ and $\delta_3$ are isomorphisms, $\mathcal{C}(f,N):\mathcal{C}(P,N)\rightarrow \mathcal{C}(K,N)$ is an epimorphism.
Thus $\mathcal{C}(\Sigma^{-1}g,N):\mathcal{C}(\Sigma^{-1}M,N)\rightarrow \mathcal{C}(\Sigma^{-1}P,N)$ is a monomorphism, and hence $$\mathcal{C}(g,\Sigma N):\mathcal{C}(M,\Sigma N)\rightarrow \mathcal{C}(P,\Sigma N)$$ is a monomorphism.

Note that  $\xi$ is closed under suspension. It follows that the triangle $$\xymatrix@C=2em{\Sigma K\ar[r]^{- \Sigma f}& \Sigma P\ar[r]^{-\Sigma g}&\Sigma M\ar[r]^{-\Sigma h}&\Sigma^{2} K}$$ is in $\xi$ with $\Sigma P\in{\mathcal{P}(\xi)}$ and $\Sigma K\in{\mathcal{P}^{\leqslant n-1}(\xi)}$. It is easy to check that $\Sigma M$ and $\Sigma N$ be objects in $\mathcal{C}$ with $\xi$-${\rm pd} \Sigma M \leq n$ and  $\xi{\rm xt}_{\xi}^{1}(\Sigma M,\Sigma N)=0$. By the forgoing proof, one can get that $\mathcal{C}(\Sigma g,\Sigma N):\mathcal{C}(\Sigma M,\Sigma N)\rightarrow \mathcal{C}(\Sigma P,\Sigma N)$ is a monomorphism. Thus $\mathcal{C}(g, N):\mathcal{C}(M, N)\rightarrow \mathcal{C}(P, N)$ is a monomorphism, and so $\delta_1:\mathcal{C}(M,N)\rightarrow \xi{\rm xt}_{\xi}^{0}(M,N)$ is an isomorphism, as desired.
\end{proof}

Let $R$ be a ring and $R$-Mod the category of left $R$-modules. Then it is clear that the class of projective left $R$-modules is generating subcategory of $R$-Mod and the class of injective $R$-modules is a cogenerating subcategory of $R$-Mod, it follows from Example \ref{Ex:3.12}(1) and Theorem \ref{thm:3.4}, we have the following corollary which contains the result of \cite[Theorem 1.1]{BM}.
\begin{cor} Let $R$ be a ring and $R$-Mod the category of left $R$-modules. Then the following are equivalent for any non-negative integer $m$:
\begin{enumerate}
\item[{\rm (1)}] $\sup\{\mathcal{G}{\rm pd}M \ | \ \textrm{for} \ \textrm{any} \ M\in R$-{\rm Mod}$\}\leq m$.

\item[{\rm (2)}]  $\sup\{\mathcal{G}{\rm id}M \ | \ \textrm{for} \ \textrm{any} \ M\in R$-{\rm Mod}$\}\leq m$.
\item[{\rm (3)}]  ${\rm spli}R={\rm silp}R\leq m$.
\end{enumerate}
Moreover, we have the following equality:
\begin{center}$\sup\{\mathcal{G}{\rm pd}M \ | \ \textrm{for} \ \textrm{any} \ M\in R$-{\rm Mod}$\}=\sup\{\mathcal{G}{\rm id}M \ | \ \textrm{for} \ \textrm{any} \ M\in R$-{\rm Mod}$\}$.\end{center}
\end{cor}

As a consequence of Example \ref{Ex:3.12}(2) and Theorem \ref{thm:3.4}, we have the following corollary which refines a result of \cite{RL1}.
\begin{cor} \label{cor:3.7}Let $\mathcal{C}$ be a triangulated category, and let $\mathcal{P}(\xi)$ be a generating subcategory of $\mathcal{C}$ and $\mathcal{I}(\xi)$ is a cogenerating subcategory of $\mathcal{C}$. Then the following are equivalent for any non-negative integer $m$:
\begin{enumerate}
\item[{\rm (1)}]  $\sup\{\xi\textrm{-}\mathcal{G}{\rm pd}M \ | \ \textrm{for} \ \textrm{any} \ M\in{\mathcal{C}}\}\leq m$.

\item[{\rm (2)}]  $\sup\{\xi\textrm{-}\mathcal{G}{\rm id}M \ | \ \textrm{for} \ \textrm{any} \ M\in{\mathcal{C}}\}\leq m$.
\item[{\rm (3)}]  $\xi\textrm{-}{\rm spli}\mathcal{C}=\xi\textrm{-}{\rm silp}\mathcal{C}\leq m$.
\end{enumerate}
Moreover, we have the following equality:
$$\sup\{\xi\textrm{-}\mathcal{G}{\rm pd}M \ | \ \textrm{for} \ \textrm{any} \ M\in{\mathcal{C}}\}=\sup\{\xi\textrm{-}\mathcal{G}{\rm id}M \ | \ \textrm{for} \ \textrm{any} \ M\in{\mathcal{C}}\}.$$
\end{cor}

\bigskip

\renewcommand\refname

{\bf References}
\begin{thebibliography}{99}
\bibitem{AS2} J. Asadollahi and S. Salarian, Tate cohomology and Gorensteinness for triangulated categories, J. Algebra 299 (2006) 480-502.
\bibitem{Bel1} A. Beligiannis, Relative homological algebra and purity in  triangulated categories, J. Algebra 227(1) (2000) 268-361.

\bibitem{BM} D. Bennis and N. Mahdou, Global Gorenstein dimensions,
Proc. Amer. Math. Soc. 138(2) (2010) 461-465.




\bibitem{GG} T. V.  Gedrich, K. W.  Gruenberg, Complete cohomological functors on groups. Topol. Appl. 25 (1987) 203-223.

\bibitem{HZZ} J. S. Hu, D. D. Zhang, P. Y. Zhou, Proper classes and Gorensteinness in extriangulated categories,
arXiv:1906.10989, 2019.
%
\bibitem{NP}  H. Nakaoka and Y. Palu,  Extriangulated categories, Hovey twin cotorsion pairs and model structures, Cahiers de Topologie et Geometrie Differentielle Categoriques, Volume LX-2 (2019) 117-193.
\bibitem{RL1} W. Ren and Z. K. Liu, Gorenstein homological dimensions for triangulated categories, J. Algebra 410 (2014) 258-276.
 \bibitem{ZZ} P. Y. Zhou and B. Zhu,  Triangulated quotient categories revisited, J. Algebra 502 (2018) 196-232.



\end{thebibliography}
\end{document}